\documentclass[11pt,reqno]{amsart}

\usepackage[T1]{fontenc}
\usepackage[utf8]{inputenc}
\usepackage{amsmath,amssymb,mathtools}
\usepackage{booktabs,longtable,array}
\usepackage{enumitem}
\usepackage{lmodern}
\usepackage{microtype}
\usepackage{tikz}
\usepackage[hidelinks]{hyperref}
\usepackage{xurl}

\makeatletter
\@namedef{subjclassname@2020}{%
  \textup{2020} Mathematics Subject Classification}
\makeatother

\hypersetup{%
  pdftitle={A configuration of 24 lines and 32 triple points on the Schur quartic with logarithmic Chern slope 14/5},
  pdfauthor={Bartosz Naskrecki and Piotr Pokora},
  pdfsubject={A line configuration on the Schur quartic},
  pdfkeywords={Schur quartic, K3 surface, line configuration, logarithmic Chern slope, ball quotient}
}

\numberwithin{equation}{section}

\newtheorem{theorem}{Theorem}[section]
\newtheorem{proposition}[theorem]{Proposition}
\newtheorem{lemma}[theorem]{Lemma}

\newtheorem{problem}[theorem]{Problem}
\theoremstyle{definition}
\newtheorem{definition}[theorem]{Definition}
\theoremstyle{remark}
\newtheorem{remark}[theorem]{Remark}

\newcommand{\PP}{\mathbb{P}}
\newcommand{\QQ}{\mathbb{Q}}
\newcommand{\CC}{\mathbb{C}}
\newcommand{\FF}{\mathbb{F}}
\newcommand{\BB}{\mathbb{B}}
\newcommand{\cbarone}{\overline{c}_{1}^{2}}
\newcommand{\cbartwo}{\overline{c}_{2}}
\newcommand{\cD}{\mathcal{D}}
\newcommand{\cC}{\mathcal{C}}

\title[A slope-\(14/5\) configuration on the Schur quartic]
{A \((24_4,32_3)\)-configuration on the Schur quartic\\
with logarithmic Chern slope \(14/5\)}

\author{Bartosz Naskr\k{e}cki}
\address{Faculty of Mathematics and Computer Science, Adam Mickiewicz
University, Uniwersytetu Pozna\'nskiego 4, 61-614 Pozna\'n, Poland}
\address{Centre for Credible AI, Warsaw University of Technology,
Rektorska 4, Warszawa, Poland}
\email{bartosz.naskrecki@amu.edu.pl}

\author{Piotr Pokora}
\address{Department of Mathematics, University of the National Education
Commission Krakow, Podchor\c a\.zych 2, 30-084 Krak\'ow, Poland}
\email{piotr.pokora@uken.krakow.pl}


\begin{document}

\begin{abstract}
We construct an arrangement \(D\) of \(24\) lines on the Schur quartic
surface, defined over \(\QQ(\sqrt{-3})\), with \(32\) ordinary triple points
as its only singularities. The associated log surface has logarithmic Chern
slope \(14/5\), which answers negatively the K3 case of the question whether
\(8/3\) bounds the slopes of transversal arrangements of rational curves.
Using Langer's orbifold Bogomolov--Miyaoka--Yau inequality, we show that
\(14/5\) is the maximal slope for arrangements of at most \(24\) smooth
rational curves on K3 surfaces satisfying a mild intersection condition.
Finally, we show that the complement of \(D\) is a ball quotient.
\end{abstract}

\maketitle

\section{Introduction}

For a smooth projective surface equipped with a reduced normal-crossing
boundary, the logarithmic Chern numbers provide a numerical measure of the
open complement.  Arrangements of curves are a useful source of such pairs
because their logarithmic Chern numbers can be calculated from their weak combinatorics; see for instance~\cite{Hirzebruch,Urzua,LafacePokora}.  In
\cite{NaskreckiPokora}, the following general question was posed.

\begin{problem}[{\cite[Problem~2.13]{NaskreckiPokora}}]
Let \(S\) be a smooth complex projective surface and let \(C\) be a connected
transversal arrangement of at least two rational curves on \(S\).  What can
be said about the logarithmic Chern slope of the associated log surface?  In
particular, is it at most \(8/3\)?
\end{problem}

The present paper concerns the specialization of this problem to rational
curve arrangements on K3 surfaces; throughout, the slope is considered only
when \(\cbartwo>0\).  Three \(32\)-line subarrangements of the Fermat quartic
attain \(8/3\), see~\cite{NaskreckiPokora}. The example below shows that this value
is not an upper bound even for line arrangements on smooth quartic K3
surfaces.

The ambient surface is not an anonymous many-line quartic.  It is Schur's
quartic~\cite{Schur}, the unique smooth complex quartic with \(64\) lines up
to projective equivalence.  The uniqueness follows from the classification
of Degtyarev, Itenberg, and Sert\"oz~\cite{DegtyarevItenbergSertoz}; explicit
equations and symmetries are discussed by Veniani~\cite{Veniani}.  Among its
\(64\) lines, \(48\) are of the first
kind and \(16\) of the second kind in the sense of~\cite[Definition~2.1]{RamsSchuett}; see
Remark~\ref{rem:kinds} and~\cite{RamsSchuettSecondKind,VenianiThesis}.  The arrangement of $48$ lines of the first kind has the weak combinatorics
\((t_2,t_3,t_4)=(144,64,0)\), and the corresponding logarithmic Chern slope
is \(64/25\); see~\cite[Proposition~5.6]{NaskreckiPokora}.

Our main observation is that these \(48\) lines split into two complementary
halves of \(24\) lines.  Each half supports \(32\) triple intersection points and has no
double or quadruple points.  The \(144\) double points of the \(48\)-line
arrangement are exactly the intersections between the two halves.  Moreover,
the halves are exchanged by a projective automorphism of the Schur quartic,
so the two labeled solutions found in the discovery computation represent a
single geometric orbit.

\begin{theorem}\label{thm:main}
Let
\begin{equation}\label{eq:schur}
 X=\{x_0^4-x_0x_1^3-x_2^4+x_2x_3^3=0\}\subset\PP^3_{\CC}
\end{equation}
be the Schur quartic, and let \(H\) be its hyperplane class.  There is a
reduced divisor
\[
 D=L_1+\cdots+L_{24}
\]
defined over \(\QQ(\sqrt{-3})\) with the following properties.
\begin{enumerate}[label=\textup{(\roman*)}]
\item The lines form a connected transversal
\((24_4,32_3)\)-configuration: the singular locus consists of \(32\)
ordinary triple points, and each line contains four of them.
\item If \(i^2=-1\) and
\[
 \alpha[x_0:x_1:x_2:x_3]=[x_0:x_1:ix_2:ix_3],
\]
then \(D\) and \(\alpha(D)\) are disjoint as sets of components, and
\(D+\alpha(D)\) is the divisor formed by the \(48\) lines of the first kind.
The two halves have \(32\) internal triple intersection points each and meet one another
in \(144\) ordinary double points.
\item One has \(D\sim6H\).
\item Let \(\pi:Y\to X\) be the blow-up of the \(32\) triple points and set
\(B=(\pi^{-1}D)_{\mathrm{red}}\).  Then
\[
 \cbarone(Y,B)=112,\qquad \cbartwo(Y,B)=40,
 \qquad E(Y,B)=\frac{14}{5}.
\]
The boundary \(B\) is semistable and \(K_Y+B\) is big.
\end{enumerate}
\end{theorem}

The proof is entirely in characteristic zero.  Appendix~\ref{app:lines}
lists the \(24\) line parametrizations, and
Appendix~\ref{app:points} lists homogeneous coordinates for every triple
point.  The finite-field and mixed-integer computations in
Section~\ref{sec:discovery} explain how the configuration was found; no
specialization or numerical-optimization claim is needed for
Theorem~\ref{thm:main}.

To the authors' knowledge, this is the first explicit transversal arrangement
of rational curves on a K3 surface whose logarithmic Chern slope exceeds
\(8/3\).

The value \(14/5\) is not accidental: \(D\) is extremal for an orbifold
Bogomolov--Miyaoka--Yau inequality, and this extremality has a
uniformization-theoretic meaning.  For a transversal arrangement
\(C=C_1+\cdots+C_n\), let \(d_i\) and \(k_i\) denote the numbers of double
points and of points of multiplicity at least three of \(C\) lying on \(C_i\).

\begin{theorem}\label{thm:intro-extremal}
\begin{enumerate}[label=\textup{(\alph*)}]
\item Let \(C\) be a transversal arrangement of \(n\le24\) smooth rational
curves on a K3 surface such that \(d_i+\tfrac12k_i\ge2\) for every \(i\).  If
\(\cbartwo>0\), then \(\cbarone/\cbartwo\le14/5\).  The arrangement \(D\) of
Theorem~\ref{thm:main} attains equality.
\item Let \(Y\) be the blow-up of \(X\) at the \(32\) triple points of \(D\),
with exceptional curves \(E_p\), and let \(D'\) be the strict transform of
\(D\).  Then \(Y\setminus D'\), with branch index \(2\) along the curves
\(E_p\), is biholomorphic to a ball quotient \(\BB^2/\Gamma\), where
\(\Gamma\subset\mathrm{PU}(2,1)\) is a non-cocompact lattice with \(24\)
cusps, one for each line of \(D\).  Consequently
\(X\setminus D\cong(\BB^2\setminus\mathcal H)/\Gamma\), where \(\mathcal H\)
is the union of the mirrors of the reflections of order two in \(\Gamma\).
\end{enumerate}
\end{theorem}

Part~(a) is Theorem~\ref{thm:extremal}.  It follows from Langer's orbifold
Bogomolov--Miyaoka--Yau inequality~\cite{Langer}, applied on the blow-up with
weight \(1\) on the strict transforms of the curves and weight \(\tfrac12\) on
the exceptional curves. For our divisor \(D\) every inequality in the proof is an
equality.  Part~(b) is Proposition~\ref{prop:ball}, which identifies this
equality as the equality case of the Kobayashi--Nakamura--Sakai
characterization of ball quotients~\cite{KNS}.  The general question remains
open: Theorem~\ref{thm:extremal} does not cover arrangements with more than
\(24\) curves or curves violating the intersection condition, and we do not
know whether \(14/5\) is the largest slope among all line arrangements on the
Schur quartic or among rational curve arrangements on K3 surfaces; see
Section~\ref{sec:questions}.

\section{Logarithmic setup}\label{sec:logarithmic-setup}

In this section we recall the notions and results that will be used in the
sequel.  We work over the complex numbers.

\begin{definition}\label{def:transversal}
Let \(S\) be a smooth complex projective surface.  A finite collection
\(\cD=\{C_1,\ldots,C_n\}\) of smooth irreducible curves on \(S\) is called a
\emph{transversal arrangement} if any two distinct curves are either disjoint
or meet with distinct tangent directions, all singular points of the reduced
divisor \(C=\sum_iC_i\) are ordinary, and \(C\) is connected.  For \(r\ge2\)
we denote by \(t_r\) the number of points at which exactly \(r\) curves from
\(\cD\) meet.
\end{definition}

Let \(\pi\colon Y\to S\) be the blow-up of \(S\) at all points of multiplicity
at least three, let \(C'\) be the strict transform of \(C\), and let
\(B:=(\pi^{-1}C)_{\mathrm{red}}\).  If \(E_p\) denotes the exceptional curve
over an ordinary \(r_p\)-fold point \(p\), then
\begin{equation}\label{eq:boundary-pullback}
 \pi^*C=C'+\sum_p r_pE_p,
 \qquad
 B=C'+\sum_pE_p.
\end{equation}
The pair \((Y,B)\) is the log surface associated with \(\cD\), and its
logarithmic Chern numbers are defined as
\[
 \cbarone(Y,B)=(K_Y+B)^2,
 \qquad
 \cbartwo(Y,B)=e(Y\setminus B)=e(S\setminus C).
\]
If \(\cbartwo(Y,B)>0\), then the logarithmic Chern slope of \((Y,B)\) is
defined as \(E(Y,B)=\cbarone(Y,B)/\cbartwo(Y,B)\).

Let us recall the following formulas for arrangements of rational curves on
K3 surfaces.

\begin{proposition}\label{prop:chern-formulas}
Let \(S\) be a K3 surface and let \(\cD\) be a transversal arrangement of
\(n\) smooth rational curves.  Then one has
\begin{align}
 \cbarone(Y,B)&=-2n+\sum_{r\ge2}(3r-4)t_r,\label{eq:c1-general}\\
 \cbartwo(Y,B)&=24-2n+\sum_{r\ge2}(r-1)t_r.\label{eq:c2-general}
\end{align}
\end{proposition}

\begin{proof}
By the adjunction formula one has \(C_i^2=-2\) for every \(i\), and hence
\[
 C^2=-2n+\sum_{r\ge2}r(r-1)t_r.
\]
Recall that \(K_Y=\pi^*K_S+\sum_pE_p=\sum_pE_p\).  Using
\eqref{eq:boundary-pullback}, we obtain
\[
 K_Y+B=\pi^*C-\sum_p(r_p-2)E_p,
\]
and since \(E_p^2=-1\), we get
\begin{align*}
 \cbarone(Y,B)&=C^2-\sum_p(r_p-2)^2
 =-2n+\sum_{r\ge2}\bigl(r(r-1)-(r-2)^2\bigr)t_r\\
 &=-2n+\sum_{r\ge2}(3r-4)t_r,
\end{align*}
which is \eqref{eq:c1-general}.  In order to show \eqref{eq:c2-general},
observe that by the inclusion-exclusion principle one has
\(e(C)=2n-\sum_{r\ge2}(r-1)t_r\), and since \(e(S)=24\), we obtain
\[
 \cbartwo(Y,B)=e(S)-e(C)=24-2n+\sum_{r\ge2}(r-1)t_r,
\]
which completes the proof.
\end{proof}

Let us now focus on line arrangements on smooth quartic surfaces.  If \(X\)
is a smooth quartic and \(p\in X\), then all lines on \(X\) passing through
\(p\) are components of the tangent plane section of \(X\) at \(p\), which has
degree four.  This means that at most four lines on \(X\) can pass through a
point.  Therefore, for a line arrangement on a smooth quartic,
Proposition~\ref{prop:chern-formulas} gives
\begin{align}
 \cbarone&=-2n+2t_2+5t_3+8t_4,\label{eq:c1-quartic}\\
 \cbartwo&=24-2n+t_2+2t_3+3t_4,\label{eq:c2-quartic}
\end{align}
and, in particular,
\begin{equation}\label{eq:G-identity}
 3\cbarone-8\cbartwo=10n-2t_2-t_3-192.
\end{equation}
It is worth pointing out that \(t_4\) does not appear
in~\eqref{eq:G-identity}; this observation motivated the discovery objective
used in Section~\ref{sec:milp}.

\begin{definition}\label{def:semistable}
A reduced simple normal crossing divisor \(B\) on a smooth projective surface
is called \emph{semistable} if every smooth rational component of \(B\) meets
the union of the other components in at least two distinct points.
\end{definition}

\begin{theorem}[Sakai, \cite{Sakai}]\label{thm:sakai}
Let \(Y\) be a smooth complex projective surface and let \(B\) be a
semistable reduced simple normal crossing divisor on \(Y\) such that
\(\kappa(Y,K_Y+B)=2\).  Then one has
\begin{equation}\label{eq:sakai}
 \cbarone(Y,B)\le3\cbartwo(Y,B).
\end{equation}
\end{theorem}

In Section~\ref{sec:log-pair} we verify that the log surface associated with
our divisor \(D\) satisfies the assumptions of Theorem~\ref{thm:sakai}.

\section{The Schur quartic and its 24-line half}\label{sec:configuration}

\subsection{The ambient surface}

Recall that \eqref{eq:schur} is the classical equation of Schur's quartic;
see~\cite{Schur,RamsSchuett}.  We start with the following simple
observation.

\begin{lemma}\label{lem:smoothness}
The surface \(X\) is smooth over \(\QQ\), and its reduction modulo \(5\) is
geometrically smooth.
\end{lemma}

\begin{proof}
Let us denote by
\[
 F=x_0^4-x_0x_1^3-x_2^4+x_2x_3^3
\]
the defining polynomial of \(X\).  Its partial derivatives are
\[
 4x_0^3-x_1^3,\quad -3x_0x_1^2,\quad
 -4x_2^3+x_3^3,\quad 3x_2x_3^2.
\]
Assume that the characteristic of the ground field is different from \(2\)
and \(3\).  Observe that the first two derivatives vanish simultaneously only
if \(x_0=x_1=0\), and the last two only if \(x_2=x_3=0\).  This means that the
four partial derivatives have no common zero in \(\PP^3\), both in
characteristic zero and in characteristic \(5\), which completes the proof.
\end{proof}

From now on we put
\[
 K=\QQ(\tau),\qquad \tau^2=-3.
\]
Recall that \(K\) is one of the three quadratic subfields of
\(\QQ(\zeta_{12})=\QQ(i,\sqrt3)\), the other two being \(\QQ(i)\) and
\(\QQ(\sqrt3)\).  It turns out that our configuration is already defined over
\(K\), so we do not need to work over the cyclotomic field.

For \(a,b,c,d\in K\) we denote by
\begin{equation}\label{eq:graph-line}
 L(a,b;c,d)=
 \operatorname{Span}_K\{(1,0,a,b),(0,1,c,d)\}
\end{equation}
the line parametrized by
\[
 [s:t]\longmapsto[s:t:as+ct:bs+dt].
\]
In Appendix~\ref{app:lines} we define the lines \(L_1,\ldots,L_{24}\) in this
form, and we set
\[
 D=\sum_{j=1}^{24}L_j.
\]

\begin{lemma}\label{lem:containment}
The line \(L(a,b;c,d)\) lies on \(X_K\) if and only if the following five
elements of \(K\) vanish:
\begin{align*}
1-a^4+ab^3,&\qquad -4a^3c+cb^3+3ab^2d,\\
-6a^2c^2+3cb^2d+3abd^2,&\qquad
-1-4ac^3+3cbd^2+ad^3,\\
&-c^4+cd^3.
\end{align*}
\end{lemma}

\begin{proof}
It is enough to substitute the above parametrization into \(F\); the five
displayed elements are exactly the coefficients of
\(s^4,s^3t,s^2t^2,st^3,t^4\), respectively.
\end{proof}

\begin{proposition}\label{prop:lines}
The \(24\) lines in Appendix~\ref{app:lines} are pairwise distinct and lie on
\(X_K\).
\end{proposition}

\begin{proof}
We apply Lemma~\ref{lem:containment} to each row of Table~\ref{tab:lines} and
we reduce all products using \(\tau^2=-3\), which gives \(24\cdot5=120\)
vanishing coefficients.  Let us present a sample computation.  The line
\(L_3\) is parametrized by
\[
 [s:t]\longmapsto
 \left[s:t:-s:\frac{1-\tau}{2}t\right],
\]
and since \(((1-\tau)/2)^3=-1\), the restriction of \(F\) to \(L_3\) is
\[
 s^4-st^3-s^4-s\left(\frac{1-\tau}{2}\right)^3t^3=0.
\]
The remaining \(115\) coefficients are recorded individually in the ancillary
JSON file described in Section~\ref{sec:ancillary}, and the standard-library
verifier reconstructs them from the printed table.  Finally, the
corresponding \(2\times4\) matrices are in reduced row-echelon form and their
parameter quadruples are pairwise distinct, hence the lines are pairwise
distinct.
\end{proof}

\subsection{Exact incidence data}

For graph lines there is a very simple incidence criterion.

\begin{lemma}\label{lem:incidence}
Let \(L_j=L(a_j,b_j;c_j,d_j)\) and
\(L_k=L(a_k,b_k;c_k,d_k)\) be distinct.  Then
\(L_j\cap L_k\ne\varnothing\) if and only if
\begin{equation}\label{eq:pair-determinant}
 \Delta_{jk}=(a_j-a_k)(d_j-d_k)-(b_j-b_k)(c_j-c_k)=0.
\end{equation}
If the determinant vanishes, then a kernel vector of
\[
 \begin{pmatrix}
 a_j-a_k&c_j-c_k\\
 b_j-b_k&d_j-d_k
 \end{pmatrix}
\]
gives the intersection point through either parametrization.
\end{lemma}

\begin{proof}
Observe that the two parametrizations agree at a non-zero vector \((s,t)\) if
and only if the above matrix has a non-trivial kernel.
\end{proof}

In Appendix~\ref{app:points} we list the \(32\) triple points together with
the three lines passing through each of them.

\begin{proposition}\label{prop:incidence}
The singular locus of \(D\) consists precisely of the \(32\) points in
Table~\ref{tab:triple-points}.  Every point is incident with the three listed
lines and with no fourth line of \(D\).  The points are pairwise distinct,
the remaining pairs of lines are skew, and every line contains four triple
points.  The divisor \(D\) is connected and all crossings are transverse.
\end{proposition}

\begin{proof}
There are \(\binom{24}{2}=276\) determinants \(\Delta_{jk}\).  Computing them
exactly in the basis \(1,\tau\), we obtain exactly \(96\) zeros, and these are
precisely the three pairs in each of the \(32\) rows of
Table~\ref{tab:triple-points}; moreover, no pair occurs in two rows.
Substituting each point into the three corresponding parametrizations, we
verify that the three lines are concurrent, and the \(32\) normalized
coordinate vectors are pairwise distinct.  The remaining \(180\) determinants
are non-zero.  This shows that there are no further intersection points and
that no fourth line of \(D\) passes through a listed point.  Let us also
point out that the machine-readable certificate contains the values of all
\(\Delta_{jk}\), not only the pattern of zeros.

Next, counting labels in Table~\ref{tab:triple-points}, we see that every
line occurs exactly four times.  In order to see that \(D\) is connected, it
is enough to exhibit a spanning tree of the intersection graph; see
Figure~\ref{fig:incidence-tree}.  Starting from \(L_1\), we attach
\[
 L_5,L_6,L_7,L_8,L_9,L_{10},L_{11},L_{12};
\]
then we attach the lines \(L_{13},L_{14},L_{15},L_{19},L_{20},L_{21}\) to
\(L_5\), and the lines \(L_{16},L_{17},L_{18},L_{22},L_{23},L_{24}\) to \(L_6\);
finally, we attach
\(L_3\) to \(L_7\), \(L_4\) to \(L_{10}\), and \(L_2\) to \(L_{19}\).  Each of
these pairs occurs in a row of Table~\ref{tab:triple-points}, and since each
of these \(23\) steps adds exactly one new vertex, we obtain a spanning tree.

Finally, the tangent direction of a line at a point is the line itself, and
distinct lines through a point have distinct directions.  Hence the three
branches at every triple point are transverse, which completes the proof.
\end{proof}

For the convenience of the reader, let us denote by \(G_D\) the graph whose
vertices are \(L_1,\ldots,L_{24}\), two vertices being adjacent precisely when
the corresponding lines intersect.  Figure~\ref{fig:incidence-tree} explains
how a triple point gives a triangle in \(G_D\), and it presents the spanning
tree used in the above proof.  Let us emphasize that the picture is only
expository; the exact incidences and coordinates are those in
Table~\ref{tab:triple-points}.

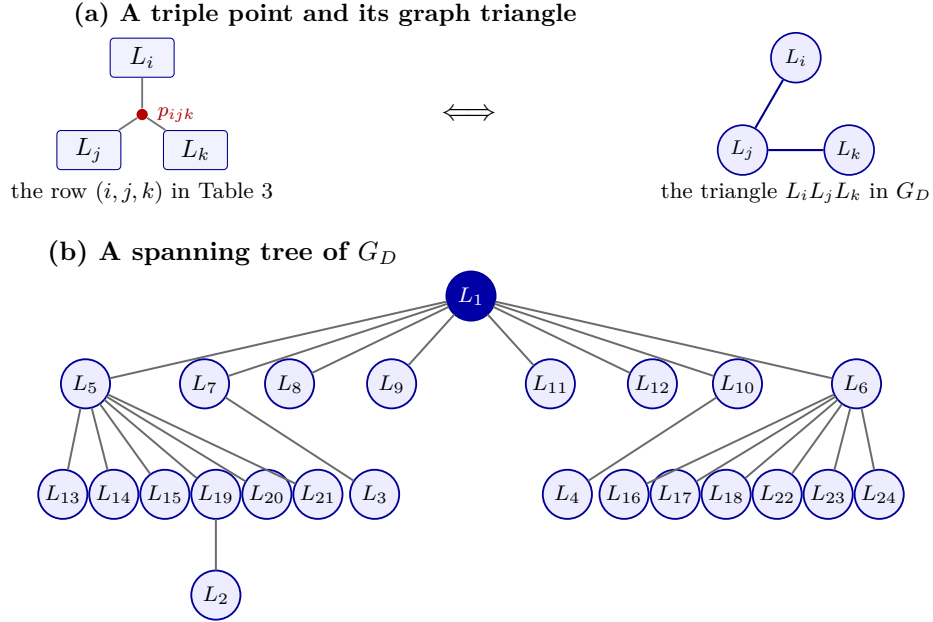
\begin{figure}[htbp]
\centering
\begin{tikzpicture}[
  scale=.94, transform shape,
  line box/.style={draw=blue!55!black, fill=blue!5, rounded corners=1.5pt,
    minimum width=9mm, minimum height=5.5mm, inner sep=1pt,
    font=\small},
  line vertex/.style={circle, draw=blue!60!black, fill=blue!7,
    line width=.65pt, minimum size=7mm, inner sep=.5pt, font=\scriptsize},
  triple vertex/.style={circle, fill=red!70!black, minimum size=4.5pt,
    inner sep=0pt},
  incidence/.style={draw=black!55, line width=.7pt},
  graph edge/.style={draw=blue!55!black, line width=.8pt}
]
  \node[font=\small\bfseries,anchor=west] at (-5.7,1.75)
    {(a) A triple point and its graph triangle};
  \node[line box] (li) at (-4.6,1.15) {$L_i$};
  \node[line box] (lj) at (-5.35,-.15) {$L_j$};
  \node[line box] (lk) at (-3.85,-.15) {$L_k$};
  \node[triple vertex,
    label={[font=\scriptsize,text=red!65!black]right:$p_{ijk}$}]
    (p) at (-4.6,.35) {};
  \draw[incidence] (li) -- (p);
  \draw[incidence] (lj) -- (p);
  \draw[incidence] (lk) -- (p);
  \node[font=\footnotesize] at (-4.6,-.75)
    {the row $(i,j,k)$ in Table~\ref{tab:triple-points}};

  \node[font=\large] at (0,.35) {$\Longleftrightarrow$};

  \node[line vertex] (gi) at (4.6,1.15) {$L_i$};
  \node[line vertex] (gj) at (3.85,-.15) {$L_j$};
  \node[line vertex] (gk) at (5.35,-.15) {$L_k$};
  \draw[graph edge] (gi) -- (gj) -- (gk) -- cycle;
  \node[font=\footnotesize] at (4.6,-.75)
    {the triangle $L_iL_jL_k$ in $G_D$};
\end{tikzpicture}

\vspace{2mm}

\begin{tikzpicture}[
  x=.75cm,y=.86cm,
  line vertex/.style={circle, draw=blue!60!black, fill=blue!7,
    line width=.7pt, minimum size=6.5mm, inner sep=.4pt, font=\scriptsize},
  root vertex/.style={line vertex, fill=blue!65!black, text=white},
  tree edge/.style={draw=black!58, line width=.75pt, line cap=round}
]
  \node[font=\small\bfseries,anchor=west] at (-7.65,5.0)
    {(b) A spanning tree of $G_D$};

  \node[root vertex] (L1) at (0,4.35) {$L_1$};

  \node[line vertex] (L5)  at (-6.8,3.0) {$L_5$};
  \node[line vertex] (L7)  at (-4.7,3.0) {$L_7$};
  \node[line vertex] (L8)  at (-3.2,3.0) {$L_8$};
  \node[line vertex] (L9)  at (-1.4,3.0) {$L_9$};
  \node[line vertex] (L11) at ( 1.4,3.0) {$L_{11}$};
  \node[line vertex] (L12) at ( 3.2,3.0) {$L_{12}$};
  \node[line vertex] (L10) at ( 4.7,3.0) {$L_{10}$};
  \node[line vertex] (L6)  at ( 6.8,3.0) {$L_6$};

  \node[line vertex] (L13) at (-7.2,1.3) {$L_{13}$};
  \node[line vertex] (L14) at (-6.3,1.3) {$L_{14}$};
  \node[line vertex] (L15) at (-5.4,1.3) {$L_{15}$};
  \node[line vertex] (L19) at (-4.5,1.3) {$L_{19}$};
  \node[line vertex] (L20) at (-3.6,1.3) {$L_{20}$};
  \node[line vertex] (L21) at (-2.7,1.3) {$L_{21}$};
  \node[line vertex] (L3)  at (-1.7,1.3) {$L_3$};

  \node[line vertex] (L4)  at ( 1.7,1.3) {$L_4$};
  \node[line vertex] (L16) at ( 2.7,1.3) {$L_{16}$};
  \node[line vertex] (L17) at ( 3.6,1.3) {$L_{17}$};
  \node[line vertex] (L18) at ( 4.5,1.3) {$L_{18}$};
  \node[line vertex] (L22) at ( 5.4,1.3) {$L_{22}$};
  \node[line vertex] (L23) at ( 6.3,1.3) {$L_{23}$};
  \node[line vertex] (L24) at ( 7.2,1.3) {$L_{24}$};

  \node[line vertex] (L2) at (-4.5,-.25) {$L_2$};

  \foreach \v in {L5,L6,L7,L8,L9,L10,L11,L12}
    \draw[tree edge] (L1) -- (\v);
  \foreach \v in {L13,L14,L15,L19,L20,L21}
    \draw[tree edge] (L5) -- (\v);
  \foreach \v in {L16,L17,L18,L22,L23,L24}
    \draw[tree edge] (L6) -- (\v);
  \draw[tree edge] (L7) -- (L3);
  \draw[tree edge] (L10) -- (L4);
  \draw[tree edge] (L19) -- (L2);
\end{tikzpicture}

\caption{Visualizing the incidence data and connectedness verification.
(a) A row \((i,j,k)\) of Table~\ref{tab:triple-points} records a triple point
incident with \(L_i,L_j,L_k\); in the line-intersection graph it contributes a
triangle.  (b) The \(23\) displayed edges form a spanning tree.  Every edge
occurs in a row of Table~\ref{tab:triple-points}; for example, \(L_1L_5\)
occurs in \((1,5,6)\), \(L_5L_{13}\) in \((5,13,19)\), and \(L_{19}L_2\) in
\((2,19,23)\).}
\label{fig:incidence-tree}
\end{figure}

\section{The two halves of the first-kind configuration}\label{sec:halves}

Let \(i^2=-1\) and consider
\[
 \alpha=\operatorname{diag}(1,1,i,i)\in\operatorname{PGL}_4(\CC).
\]
Since every monomial of~\eqref{eq:schur} involving \(x_2,x_3\) has degree four
in \((x_2,x_3)\), we have
\[
 F(x_0,x_1,ix_2,ix_3)=F(x_0,x_1,x_2,x_3),
\]
and hence \(\alpha\) is a projective automorphism of \(X\).

\begin{proposition}\label{prop:halves}
The divisors \(D\) and \(D^*=\alpha(D)\) have no common component.  Their
union is the configuration \(\cC_2\) of the \(48\) lines of the first kind
on the Schur quartic, and the incidence data decompose as follows:
\begin{center}
\begin{tabular}{@{}lrrrr@{}}
\toprule
configuration & \(n\) & \(t_2\) & \(t_3\) & \(t_4\)\\
\midrule
\(D\) & 24 & 0 & 32 & 0\\
\(D^*\) & 24 & 0 & 32 & 0\\
\(D+D^*\) & 48 & 144 & 64 & 0\\
\bottomrule
\end{tabular}
\end{center}
The \(64\) triple points of \(\cC_2\) are the two disjoint sets of internal
triple points, while all \(144\) double points are intersections of one line
from each half.  Moreover, \(\alpha(D)=D^*\) and \(\alpha(D^*)=D\), so the
two halves are equivalent under \(\operatorname{Aut}(X)\).
\end{proposition}

\begin{proof}
Applying \(\alpha\) to the graph parameters in Appendix~\ref{app:lines}, we
multiply \(a,b,c,d\) by \(i\), and we obtain \(24\) new graph lines.  Now we
compare them with the standard list of the \(48\) first-kind
lines~\cite[Section~4.2]{VenianiThesis}.  In order to do so, we use the same
equation~\eqref{eq:schur} and the unique \(2\times4\) RREF representative of
each line, with the pivot charts ordered as in Section~\ref{sec:discovery}:
\[
 (1,2),(1,3),(1,4),(2,3),(2,4),(3,4).
\]
After putting Veniani's list into this normalization, the \(48\) matrices of
\(\cC_2\) are exactly the \((1,2)\)-chart block of the whole \(64\)-line list.
The finite-field Magma check reconstructs the same ordered list, with chart
sizes
\[
 48,9,3,3,1,0,
\]
and the machine-readable MILP record uses these indices.  In this indexing one
has
\[
 D=S_+=\{0,\ldots,11,24,\ldots,35\},\qquad
 D^*=S_-=\{12,\ldots,23,36,\ldots,47\},
\]
so \(D+D^*\) is the first \(48\)-line block, namely \(\cC_2\).

The exact pairwise intersections over \(\QQ(i,\sqrt3)\) give the profile
presented in the table.  More precisely, we obtain \(32\) triple points
internal to each half and \(144\) mixed double points, and there are no other
intersection points.  This computation is a part of the standard-library
certificate.  Finally, observe that \(\alpha^2\) multiplies \((x_2,x_3)\) by
\(-1\), and a direct comparison of the graph parameters shows that
\(\alpha^2(D)=D\).  Hence \(\alpha\) exchanges the two halves, which completes
the proof.
\end{proof}

Proposition~\ref{prop:halves} explains the two labeled \(24\)-subsets returned
by the discovery computation: they are isomorphic as incidence structures,
and they lie in one orbit of projective automorphisms.  Their union is exactly
the \(48\)-line arrangement of slope
\[
 \frac{-2\cdot48+2\cdot144+5\cdot64}
 {24-2\cdot48+144+2\cdot64}=\frac{64}{25}.
\]

\begin{remark}\label{rem:kinds}
Let us recall that a line \(\ell\subset X\) is of the second kind if it meets
every smooth residual cubic of the pencil of planes through \(\ell\) in
inflection points, and it is of the first kind
otherwise~\cite[Definition~2.1]{RamsSchuett}.  Put
\(\varphi(x,y)=x(x^3-y^3)\).  The \(16\) lines of \(X\) which do not belong to
\(\cC_2\) are of the form \(\{f(x_0,x_1)=g(x_2,x_3)=0\}\), where \(f,g\) are
linear factors of \(\varphi\).  In other words, \(\cC_2\) consists of the lines
that are not contained in the eight planes \(f(x_0,x_1)=0\), \(g(x_2,x_3)=0\),
and each of these planes cuts out four concurrent lines.  We claim that these
\(16\) lines are of the second kind.  Indeed, the Klein four-group permuting the
roots of \(\varphi\) lifts to matrices \(h\) with \(\varphi\circ h=\varphi\),
so the automorphisms \(\operatorname{diag}(h,h')\) of \(X\) act transitively on
the \(16\) lines.  For \(\ell=\{x_0=x_2=0\}\), the residual cubics in the planes
\(x_2=\lambda x_0\) are the diagonal cubics
\((1-\lambda^4)x_0^3-x_1^3+\lambda x_3^3=0\), whose Hessian is a multiple of
\(x_0x_1x_3\), and hence it vanishes along \(\ell\).  On the other hand, a
direct computation of the resultant in~\cite[Definition~3.3.10]{VenianiThesis}
shows that the \(48\) lines of \(\cC_2\), which are all of type \((4,6)\), are
of the first kind; see also~\cite[Section~4.1]{VenianiThesis}.
\end{remark}

\section{The divisor class and the log surface}\label{sec:log-pair}

The regularity of the incidences has the following divisorial consequence.

\begin{proposition}\label{prop:divisor-class}
Let \(H\) be the hyperplane class on \(X\).  Then
\[
 D\sim6H.
\]
\end{proposition}

\begin{proof}
Recall that every line on a K3 surface is a \((-2)\)-curve.  Since the \(32\)
triple points give \(32\cdot\binom{3}{2}=96\) pairs of intersecting lines, and
each such pair meets transversally at exactly one point, we obtain
\[
 D^2=24\cdot(-2)+2\cdot96=144.
\]
Moreover, \(D\cdot H=24\) and \(H^2=4\).  Let us consider \(Z=D-6H\).  Then
\[
 Z\cdot H=0,
 \qquad
 Z^2=D^2-12D\cdot H+36H^2=0.
\]
By the Hodge index theorem, the intersection form is negative definite on
\(H^\perp\subset\operatorname{NS}(X)_\mathbb{R}\), so \(Z\) is numerically
trivial.  Since \(\operatorname{Pic}^0(X)=0\) and the N\'eron--Severi group of
a K3 surface is torsion-free, numerical and linear equivalence coincide, and
hence \(D\sim6H\).
\end{proof}

Now we are ready to prove the logarithmic part of Theorem~\ref{thm:main}.
Let \(p_1,\ldots,p_{32}\) be the triple points, let \(E_1,\ldots,E_{32}\) be
the exceptional curves of \(\pi\), and let \(D'\) be the strict transform of
\(D\).  Then
\[
 \pi^*D=D'+3\sum_{p=1}^{32}E_p,
 \qquad
 B=D'+\sum_{p=1}^{32}E_p,
 \qquad
 K_Y=\sum_{p=1}^{32}E_p.
\]
Using Proposition~\ref{prop:divisor-class}, we get
\begin{equation}\label{eq:log-canonical-class}
 K_Y+B=\pi^*D-\sum_{p=1}^{32}E_p
       =6\pi^*H-\sum_{p=1}^{32}E_p,
\end{equation}
and therefore
\begin{equation}\label{eq:c1-direct}
 \cbarone(Y,B)=(K_Y+B)^2=36H^2-32=112.
\end{equation}
In order to compute the second logarithmic Chern number, observe that by the
inclusion-exclusion principle on \(X\) one has
\[
 e(D)=24\,e(\PP^1)-2\cdot32=48-64=-16,
\]
since each triple point is counted three times in the sum of the Euler
characteristics of the components and it must be counted once.  Therefore
\begin{equation}\label{eq:c2-direct}
 \cbartwo(Y,B)=e(X\setminus D)=24-(-16)=40,
\end{equation}
and combining~\eqref{eq:c1-direct} with~\eqref{eq:c2-direct}, we obtain
\[
 E(Y,B)=\frac{112}{40}=\frac{14}{5}>\frac83.
\]
It is worth pointing out that the same values follow from substituting
\((n,t_2,t_3,t_4)=(24,0,32,0)\) into~\eqref{eq:c1-quartic}
and~\eqref{eq:c2-quartic}.

It remains to verify the assumptions of Theorem~\ref{thm:sakai}.  Since
\(t_2=0\), the components of \(D'\) are pairwise disjoint, and hence \(B\) is a
simple normal crossing divisor.  Each component of \(D'\) meets four
exceptional curves, and each \(E_p\) meets three components of \(D'\), so
every rational component of \(B\) meets the other components in at least
three points, i.e., \(B\) is semistable.  Moreover,
by~\eqref{eq:log-canonical-class} we have \((K_Y+B)\cdot D'_j=6-4=2\) for every
component \(D'_j\) of \(D'\) and \((K_Y+B)\cdot E_p=1\), so the effective divisor
\(K_Y+B=D'+2\sum_pE_p\) is nef.  Since \((K_Y+B)^2=112>0\), it is also big, and
hence \(\kappa(Y,K_Y+B)=2\).  This completes the proof of
Theorem~\ref{thm:main}.

\section{Extremality of the slope and a ball quotient}\label{sec:ball}

In this section we explain why the value \(14/5\) is extremal for our divisor
\(D\).  Our main tool is the orbifold version of the Bogomolov--Miyaoka--Yau
inequality due to Langer.  Let us recall the necessary notions.  Let \(Z\) be
a smooth complex projective surface and let \(\Delta=\sum_ja_j\Delta_j\) be a
simple normal crossing divisor with \(a_j\in[0,1]\), so that \((Z,\Delta)\) is
log canonical.  The orbifold Euler number of \((Z,\Delta)\) is defined as
\[
 e_{\mathrm{orb}}(Z,\Delta)=e(Z\setminus\operatorname{Supp}\Delta)
 +\sum_j(1-a_j)\,e(\Delta_j^\circ)
 +\sum_{j<k}(1-a_j)(1-a_k)\,\#(\Delta_j\cap\Delta_k),
\]
where \(\Delta_j^\circ=\Delta_j\setminus\bigcup_{k\ne j}\Delta_k\).

\begin{theorem}[Langer, \cite{Langer}]\label{thm:langer}
Let \((Z,\Delta)\) be as above and assume that \(K_Z+\Delta\) is nef.  Then one
has
\begin{equation}\label{eq:langer}
 (K_Z+\Delta)^2\le3\,e_{\mathrm{orb}}(Z,\Delta).
\end{equation}
\end{theorem}

Our first result in this section is the following Hirzebruch-type inequality
for arrangements of rational curves on K3 surfaces.

\begin{theorem}\label{thm:extremal}
Let \(S\) be a K3 surface and let \(C=C_1+\cdots+C_n\) be a transversal
arrangement of smooth rational curves on \(S\), not necessarily connected.
For every \(i\), denote by \(d_i\) the number of double points and by \(k_i\)
the number of points of multiplicity at least three of \(C\) lying on \(C_i\).
Assume that
\begin{equation}\label{eq:nef-condition}
 d_i+\tfrac12k_i\ge2\qquad\text{for all } i\in\{1,\ldots,n\}.
\end{equation}
Then one has
\begin{equation}\label{eq:orbifold-hirzebruch}
 t_2+\sum_{r\ge3}\Bigl(\frac94-\frac r2\Bigr)t_r\ \ge\ 4n-72,
\end{equation}
and
\begin{equation}\label{eq:slope-bound}
 5\cbarone-14\cbartwo\ \le\ 2(n-24)-\sum_{r\ge4}(r-3)t_r .
\end{equation}
In particular, if \(n\le24\) and \(\cbartwo>0\), then
\(E(Y,B)\le14/5\), and if equality holds, then \(n=24\), \(t_r=0\) for
\(r\ge4\), and~\eqref{eq:orbifold-hirzebruch} is an equality.  Our divisor
\(D\) satisfies~\eqref{eq:nef-condition} with equality for every line, and
\(E(Y,B)=14/5\).
\end{theorem}

\begin{proof}
Let \(\pi\colon Y\to S\), \(C'\) and \(E_p\) be as in
Section~\ref{sec:logarithmic-setup}, and put \(E=\sum_pE_p\) and
\(s=\sum_{r\ge3}t_r\).  Let us consider the pair \((Y,\Delta)\) with
\(\Delta=C'+\tfrac12E\).  Since \(\Delta\) is a simple normal crossing divisor
with coefficients in \(\{\tfrac12,1\}\), the pair \((Y,\Delta)\) is log
canonical.  Recall that \(K_Y=E\), hence
\[
 L:=K_Y+\Delta=C'+\tfrac32E .
\]
First, we show that \(L\) is nef.  Observe that \(L\cdot E_p=r_p-\tfrac32>0\).
Moreover, \(C_i'^2=-2-k_i\), the curve \(C_i'\) meets the other components of
\(C'\) in \(d_i\) points, and it meets \(E\) in \(k_i\) points, hence
\(L\cdot C_i'=d_i+\tfrac12k_i-2\ge0\) by~\eqref{eq:nef-condition}.  Since
\(L\) is effective, it is nef.

Now we compute both sides of~\eqref{eq:langer}.  Since
\(C'^2=C^2-\sum_pr_p^2=-2n+2t_2-\sum_{r\ge3}rt_r\),
\(C'\cdot E=\sum_{r\ge3}rt_r\) and \(E^2=-s\), we get
\[
 L^2=-2n+2t_2+\sum_{r\ge3}\Bigl(2r-\frac94\Bigr)t_r .
\]
In order to compute \(e_{\mathrm{orb}}(Y,\Delta)\), observe that the components
of \(C'\) have coefficient \(1\), so they and all intersection points lying on
them do not contribute; each \(E_p\) contributes
\(\tfrac12e(E_p^\circ)=\tfrac12(2-r_p)\); and
\(Y\setminus\operatorname{Supp}\Delta\cong S\setminus C\) has Euler number
\(24-2n+\sum_{r\ge2}(r-1)t_r\) by the proof of
Proposition~\ref{prop:chern-formulas}.  Hence
\[
 e_{\mathrm{orb}}(Y,\Delta)=24-2n+t_2+\sum_{r\ge3}\frac r2\,t_r ,
\]
and Theorem~\ref{thm:langer} gives exactly~\eqref{eq:orbifold-hirzebruch}.

In order to obtain~\eqref{eq:slope-bound}, recall that
Proposition~\ref{prop:chern-formulas} gives
\[
 5\cbarone-14\cbartwo=18n-336-\sum_{r\ge2}(6-r)t_r.
\]
Observe that
\[
 \sum_{r\ge2}(6-r)t_r=4\Bigl[t_2+\sum_{r\ge3}\Bigl(\frac94-\frac r2\Bigr)t_r\Bigr]
 +\sum_{r\ge4}(r-3)t_r,
\]
so~\eqref{eq:orbifold-hirzebruch} implies~\eqref{eq:slope-bound}.  If
\(n\le24\), then the right-hand side of~\eqref{eq:slope-bound} is non-positive,
and it vanishes only in the stated case.  Finally, for our divisor \(D\) one
has \(n=24\), \(t_3=32\), \(d_i=0\) and \(k_i=4\), and both sides
of~\eqref{eq:orbifold-hirzebruch} are equal to \(24\), which completes the
proof.
\end{proof}

\begin{remark}\label{rem:extremal}
Let us observe that condition~\eqref{eq:nef-condition} is exactly the nefness
of \(K_Y+C'+\tfrac12E\) along \(C'\).  It is worth pointing out that this
condition cannot simply be omitted in our argument.  Indeed, if we add to an
arrangement a smooth rational curve passing through three of its double
points and meeting it nowhere else, then~\eqref{eq:nef-condition} fails for
the new curve, and \(5\cbarone-14\cbartwo\) increases by \(21\).  On the other
hand, for \(n>24\) inequality~\eqref{eq:slope-bound} only gives
\[
 E(Y,B)\le\frac{14}{5}+\frac{2(n-24)-\sum_{r\ge4}(r-3)t_r}{5\cbartwo}.
\]
\end{remark}

It turns out that the equality attained by our divisor \(D\) in
Theorem~\ref{thm:extremal} is exactly the equality case of the numerical
characterization of ball quotients due to Kobayashi, Nakamura and Sakai.  For
the convenience of the reader, let us recall their setting.  Let \(V\) be a
compact normal complex surface and let \(\Theta=\sum_i(1-1/b_i)\Theta_i\),
with \(b_i\in\{2,3,\ldots,\infty\}\), be a \(\QQ\)-divisor such that
\((V,\Theta)\) has at worst log canonical singularities and log Kodaira
dimension \(2\).  Its log canonical model \((V'',\Theta'')\) is obtained by
successive contractions of log-exceptional curves of the first kind, i.e.,
irreducible curves \(C\) with \(C^2<0\) and \((K_V+\Theta)\cdot C<0\), and then
of the second kind, i.e., irreducible curves \(C\) with \(C^2<0\) and
\((K_V+\Theta)\cdot C=0\).  Denote by \(\operatorname{LCS}(V'',\Theta'')\) the
set of singular points of \((V'',\Theta'')\) which are log canonical but not
log terminal, and put
\[
 V_0''=V''\setminus\Bigl(\bigcup_{b_i=\infty}\Theta_i''\cup
 \operatorname{LCS}(V'',\Theta'')\Bigr),
 \qquad \Theta_{0i}''=\Theta_i''\cap V_0''.
\]

\begin{theorem}[Kobayashi--Nakamura--Sakai, \cite{KNS}]\label{thm:kns}
In the above setting one has
\begin{equation}\label{eq:kns}
 (K_{V''}+\Theta'')^2\le3\Bigl(e(V_0'')+\sum_i\Bigl(\frac1{b_i}-1\Bigr)
 \bigl(e(\Theta_{0i}'')-d_i\bigr)+\sum_p\Bigl(\frac1{|G_p|}-1\Bigr)\Bigr),
\end{equation}
where \(d_i\) is the number of singular points of \((V'',\Theta'')\) lying on
\(\Theta_{0i}''\), and the last sum runs over the log terminal singular points
\(p\) of \((V'',\Theta'')\), with \(G_p\) the local fundamental group of \(p\) in
the sense of orbifolds.  Moreover, equality holds if and only if \(V_0''\),
with branch indices \(b_i\) along \(\Theta_{0i}''\), is biholomorphic to a ball
quotient \(\BB^2/\Gamma\) for a discrete subgroup
\(\Gamma\subset\operatorname{Aut}(\BB^2)=\mathrm{PU}(2,1)\).
\end{theorem}

Now we are ready to present our second result in this section.

\begin{proposition}\label{prop:ball}
Let \(Y\), \(D'\) and \(E_p\) be as in Section~\ref{sec:log-pair} and put
\(\Delta=D'+\tfrac12\sum_pE_p\).  The log canonical model \((X'',\Delta'')\) of
\((Y,\Delta)\) is obtained by contracting the \(24\) curves \(D_i'\) to points
\(q_1,\ldots,q_{24}\), and \(X''\setminus\{q_1,\ldots,q_{24}\}\cong
Y\setminus D'\), with branch index \(2\) along the curves \(E_p\), is
biholomorphic to a ball quotient \(\BB^2/\Gamma\), where
\(\Gamma\subset\mathrm{PU}(2,1)\) is a non-cocompact lattice whose cusps
correspond to \(q_1,\ldots,q_{24}\).  In particular,
\(X\setminus D\cong(\BB^2\setminus\mathcal H)/\Gamma\), where \(\mathcal H\)
is the union of the mirrors of the reflections of order two in \(\Gamma\).
\end{proposition}

\begin{proof}
Our aim is to apply Theorem~\ref{thm:kns} to the pair \((Y,\Delta)\).  First
of all, \((Y,\Delta)\) is log canonical.  Since \(\pi^*D=D'+3\sum_pE_p\) and
\(D\sim6H\), we have
\[
 L:=K_Y+\Delta=D'+\tfrac32\textstyle\sum_pE_p=6\pi^*H-\tfrac32\sum_pE_p,
\]
and hence \(L^2=144-72=72\), \(L\cdot E_p=\tfrac32\) and
\(L\cdot D_i'=6-\tfrac32\cdot4=0\).  Let \(C\subset Y\) be any other irreducible
curve.  Since \(D\) is ample, we have \(C\cdot(D'+3\sum_pE_p)=\pi_*C\cdot D>0\),
and therefore \(L\cdot C=C\cdot D'+\tfrac32C\cdot\sum_pE_p>0\).  This means
that \(L\) is nef and big.  In particular, \((Y,\Delta)\) has log Kodaira
dimension \(2\), there are no log-exceptional curves of the first kind, and
the log-exceptional curves of the second kind are exactly the pairwise
disjoint \((-6)\)-curves \(D_i'\).  Contracting them, we obtain the log
canonical model \(\mu\colon Y\to X''\) with \(L=\mu^*(K_{X''}+\Delta'')\), and
the points \(q_i\) are log canonical but not log terminal.

Now we compute the right-hand side of~\eqref{eq:kns}.  In the notation of
Theorem~\ref{thm:kns}, observe that
\(X_0''=X''\setminus\{q_1,\ldots,q_{24}\}\cong Y\setminus D'\), and hence
\(e(X_0'')=e(Y)-48=56-48=8\).  The branch curves are the \(32\) curves \(E_p\)
with \(b=2\); each of them meets \(D'\) in three points, so
\(e(E_p\setminus D')=-1\) and \(d_p=0\).  Finally, \((X'',\Delta'')\) has no
log terminal singular points.  Therefore~\eqref{eq:kns} reads
\[
 72=(K_{X''}+\Delta'')^2\le3\Bigl(8+32\cdot\Bigl(\frac12-1\Bigr)\cdot(-1)\Bigr)=72 .
\]
By the equality part of Theorem~\ref{thm:kns}, \(X_0''\), with branch index
\(2\) along the curves \(E_p\setminus D'\), is biholomorphic to \(\BB^2/\Gamma\)
for a discrete subgroup \(\Gamma\subset\mathrm{PU}(2,1)\).  Since \(X_0''\) is
the complement of \(24\) points in the compact surface \(X''\), the group
\(\Gamma\) is a non-cocompact lattice whose cusps correspond to the points
\(q_i\), and its orbifold Euler characteristic is
\(e_{\mathrm{orb}}(Y,\Delta)=24\).  Finally, since the curves \(E_p\) are
pairwise disjoint and \(X_0''\) is smooth, the only non-trivial isotropy of the
orbifold lies along the branch curves.  Hence \(\Gamma\) acts freely on the
complement of the preimage \(\mathcal H\) of \(\bigcup_pE_p\), which is the
union of the mirrors of the reflections of order two in \(\Gamma\), and we
obtain
\[
 (\BB^2\setminus\mathcal H)/\Gamma\cong Y\setminus\Bigl(D'\cup\bigcup_pE_p\Bigr)
 \cong X\setminus D,
\]
which completes the proof.
\end{proof}

\section{Discovery computation}\label{sec:discovery}

This section records how the example was located.  Its role is historical
and reproducibility-oriented.  The proof in
Sections~\ref{sec:configuration}--\ref{sec:ball} does not depend on the
finite-field model, a lifting argument, or a claim of certified numerical
optimization.

\subsection{Complete line reconstruction in characteristic five}%
\label{sec:finite-field}

The reduction of~\eqref{eq:schur} modulo \(5\) is smooth by
Lemma~\ref{lem:smoothness}.  Over \(\FF_{25}\), every projective line has a
unique \(2\times4\) reduced row-echelon matrix.  The six pivot charts contain
\[
 q^4+q^3+2q^2+q+1=407{,}526
 \qquad(q=25)
\]
lines.  Direct substitution of a generic chart line into the quartic and
testing all chart parameters gave the following counts:
\begin{center}
\begin{tabular}{@{}rrrrrrr@{}}
\toprule
pivot chart & \((1,2)\) & \((1,3)\) & \((1,4)\) & \((2,3)\) & \((2,4)\) & \((3,4)\)\\
\midrule
contained lines & 48 & 9 & 3 & 3 & 1 & 0\\
\bottomrule
\end{tabular}
\end{center}
Thus the scan produced \(64\) distinct \(\FF_{25}\)-rational lines.  This is
also a complete list of geometric lines: Rams and Sch\"utt prove that a
geometrically smooth quartic in characteristic different from \(2\) and
\(3\) contains at most \(64\) lines
by~\cite[Theorem~1.2]{RamsSchuett}.

Exact row-space intersections gave the full profile
\[
 (t_2,t_3,t_4)=(336,64,8).
\]
The full \(64\)-line arrangement therefore has
\[
 (\cbarone,\cbartwo,E)=(928,384,29/12).
\]

The finite-field computation supplied an incidence structure on which to
search.  It did not prove that a selected finite-field subset lifted with the
same incidence data.  After discovery, the lines in Appendix~\ref{app:lines}
and all their incidences were reconstructed and checked directly in
characteristic zero.  The finite-field-to-characteristic-zero index map in
the ancillary data is documentary and is not part of the proof.

\subsection{The mixed-integer model}\label{sec:milp}

Let the \(64\) ambient lines be indexed by \(0,\ldots,63\), and introduce a
binary variable \(x_i\) for each line.  At an ambient double point
\(p=\{i,j\}\), a binary variable \(y_p=x_ix_j\) is imposed by the standard
binary-product linearization~\cite{McCormick}:
\begin{equation}\label{eq:double-linearization}
 y_p\le x_i,\qquad y_p\le x_j,\qquad
 y_p\ge x_i+x_j-1.
\end{equation}
At an ambient point \(p\) of multiplicity \(r\in\{3,4\}\), binary variables
\(z_{p,k}\), \(0\le k\le r\), satisfy
\begin{equation}\label{eq:one-hot}
 \sum_{k=0}^r z_{p,k}=1,
 \qquad
 \sum_{k=0}^r k z_{p,k}=\sum_{i\in p}x_i.
\end{equation}
Hence
\begin{align*}
 n&=\sum_i x_i,\\
 t_2&=\sum_{p:\,|p|=2}y_p+\sum_{p:\,|p|\ge3}z_{p,2},\\
 t_3&=\sum_{p:\,|p|\ge3}z_{p,3},
 \qquad
 t_4=\sum_{p:\,|p|=4}z_{p,4}.
\end{align*}
The model has \(696\) binary variables.  The condition that the slope have a
positive denominator is encoded as the integer inequality
\begin{equation}\label{eq:c2-positive}
 \cbartwo=24-2n+t_2+2t_3+3t_4\ge1.
\end{equation}

Connectedness was not encoded.  The solver therefore optimized over a larger
class than the class of transversal arrangements.  This was intentional for
candidate discovery; each returned candidate was checked for connectedness
afterward.

Equation~\eqref{eq:G-identity} gives the integral discovery objective
\[
 G=10n-2t_2-t_3.
\]
For direct slope optimization, the implementation used Dinkelbach
iteration~\cite{Dinkelbach}:
for \(\lambda=p/q\), maximize
\[
 q\cbarone-p\cbartwo.
\]
The implemented objective omits the constant term \(24\) in \(\cbartwo\), which
shifts the objective value but does not change its maximizers.
The rational values of \(p/q\) and the objective coefficients were formed
exactly in Python, but each MILP was solved numerically by HiGHS~\cite{HiGHS}.  Exact
rational model coefficients do not turn the branch-and-bound run into an
exact proof.

The successive incumbents have the following geometric interpretation.
\begin{center}
\begin{tabular}{@{}rccccr@{}}
\toprule
lines & \((t_2,t_3,t_4)\) & \(\cbarone\) & \(\cbartwo\) & slope & role\\
\midrule
64 & \((336,64,8)\) & 928 & 384 & \(29/12\) & full configuration\\
48 & \((144,64,0)\) & 512 & 200 & \(64/25\) & first-kind lines\\
24 & \((0,32,0)\) & 112 & 40 & \(14/5\) & one half\\
\bottomrule
\end{tabular}
\end{center}
The implementation initialized \(\lambda\) at \(-2N=-128\), where
\(N=64\).  The meaningful updates are
\[
 \frac{29}{12}\longrightarrow\frac{64}{25}
 \longrightarrow\frac{14}{5}.
\]

\subsection{The two labeled solver vectors and the no-good logic}

In the finite-field RREF labeling, the run returned
\begin{align*}
 S_+&=\{0,1,\ldots,11,24,25,\ldots,35\},\\
 S_-&=\{12,13,\ldots,23,36,37,\ldots,47\}.
\end{align*}
They are complementary subsets of the first \(48\) indices.  Under the exact
characteristic-zero index correspondence, \(S_+\) is the divisor \(D\) in
Appendix~\ref{app:lines}, while \(S_-\) is \(\alpha(D)\).  Thus
Proposition~\ref{prop:halves} gives their geometric relation.

The enumeration first imposed the optimal-level equality.  Since the implementation stores
\(\cbartwo^\circ=\cbartwo-24\), the slope-\(14/5\) level is
\begin{equation}\label{eq:optimal-level}
 5\cbarone-14\cbartwo^\circ=14\cdot24=336.
\end{equation}
For a binary vector \(S\), the standard no-good cut
\begin{equation}\label{eq:no-good}
 \sum_{i\in S}(1-x_i)+\sum_{i\notin S}x_i\ge1
\end{equation}
excludes exactly that vector.  The historical run
solved~\eqref{eq:optimal-level}, found \(S_+\), added~\eqref{eq:no-good}, found
\(S_-\), added its no-good cut, and then reported infeasibility of the
\emph{optimal-level model}.  Nonoptimal subsets remain feasible in the
original model.

The solver output is evidence about the finite search, not a formal
branch-and-bound certificate.  The archived historical JSON did not record
all software versions, tolerances, node logs, or a proof object; consequently
we do not use it to assert an independently checkable classification of all
maximizers.  A reproducibility run with explicit settings and complete logs
is listed in Section~\ref{sec:ancillary}.  Both returned vectors are verified
afterward by exact arithmetic, which proves their profiles but does not by
itself certify numerical global optimality.

\section{Reproducibility package}\label{sec:ancillary}

The computational material is collected in a dedicated reproducibility
repository~\cite{AncillaryRepository}.  The repository contains only the
scripts, input data, generated table rows, and normalized check records needed
to reproduce the computational claims below.

The theorem certificate is designed to be checked without an optimizer or a
computer algebra system.  From the repository root, run
\begin{center}
\texttt{python3 code/verify\_schur\_24\_certificate.py}
\end{center}
with \texttt{Python} \(3.11\) or later.  The script uses only the standard library and
implements arithmetic in the basis
\(1,i,\sqrt3,i\sqrt3\).  It reconstructs the \(24\) printed lines and checks:
\begin{itemize}[leftmargin=2em]
\item all \(120\) restricted-quartic coefficients;
\item all \(276\) pair determinants, including their exact values;
\item all \(32\) normalized triple-point coordinates and incident labels;
\item point distinctness, absence of fourth selected lines, the \(180\) skew
pairs, regularity, transversality data, and connectedness;
\item the \(48=24+24\) decomposition and its \(144\) cross double points;
\item the divisor and logarithmic Chern-number identities.
\end{itemize}

The principal files are:
\begingroup
\footnotesize
\begin{longtable}{@{}>{\raggedright\arraybackslash}p{0.60\textwidth}%
                       >{\raggedright\arraybackslash}p{0.32\textwidth}@{}}
\toprule
file & role\\
\midrule
\endfirsthead
\toprule
file & role\\
\midrule
\endhead
\bottomrule
\endfoot
\path{verify_reproducibility.py}
& bundle manifest check and exact certificate smoke test\\
\path{code/verify_schur_24_certificate.py}
& standard-library exact verifier\\
\path{data/SCHUR-24/schur_24_exact_certificate.json}
& all line, coefficient, determinant, point, and decomposition data\\
\path{generated/schur_24_triple_points.tex}
& generated rows of Table~\ref{tab:triple-points}\\
\path{code/magma/verify_veniani64_counterexample_char0.m}
& second implementation in Magma over \(\QQ(\zeta_{12})\)\\
\path{checks/magma/veniani64_char0_magma_audit.txt}
& normalized output from Magma V2.28-3\\
\path{code/magma/verify_veniani64_gf25_fast.m}
& exhaustive finite-field RREF scan and incidence check\\
\path{checks/magma/veniani64_gf25_magma_audit.txt}
& normalized finite-field output from Magma V2.28-3\\
\path{code/reproduce_veniani64_milp.py}
& controlled one-thread numerical reproduction\\
\path{checks/veniani-64_milp_reproduction.json}
& solver settings, bounds, gaps, nodes, vectors, and statuses\\
\path{MANIFEST.sha256}
& SHA-256 hashes for the reproducibility payload\\
\end{longtable}
\addtocounter{table}{-1}
\endgroup

The numerical discovery record is regenerated from the repository root with
\begin{center}
\texttt{python3 code/reproduce\_veniani64\_milp.py}.
\end{center}
It uses the following software~\cite{SciPy,HiGHS}:
\[
 \text{\texttt{Python} 3.11.5},\quad \text{\texttt{NumPy} 1.26.3},\quad
 \text{\texttt{SciPy} 1.12.0},\quad \text{\texttt{HiGHS} 1.2.0}.
\]
The recorded settings are one thread, random seed zero, presolve enabled,
primal and dual feasibility tolerances \(10^{-7}\), MIP feasibility
tolerance \(10^{-6}\), and zero requested relative MIP gap.  The structured
record stores the result status, primal bound, dual bound, node count, reported
gap, and selected vectors.  The recorded one-thread run took \(800.163511\)
seconds.  Every successful solve reported equal primal and dual
bounds and zero gap.  After the two optimal-level vectors were excluded, the
final solve returned HiGHS status~8 (infeasible).  These disclosures make the
numerical run reproducible; they do not convert
floating-point branch-and-bound into an exact proof
certificate.

The older Zenodo record~\cite{NaskreckiData} contains the upstream quartic
scripts used to initiate the reconstruction.  The theorem certificate, Magma
cross-checks, finite-field MILP input data, and recorded solver evidence are
distributed in the reproducibility repository~\cite{AncillaryRepository}.

\section{Further questions}\label{sec:questions}

The decomposition \(\cC_2=D+\alpha(D)\) suggests questions that are not
visible from the weak combinatorics alone.  First, it would be useful to
describe the two-coloring of the \(48\) first-kind lines intrinsically, for
example through a character of an automorphism subgroup or a lattice class,
rather than through the field and parametrization used here.  The two halves
are projectively equivalent by Proposition~\ref{prop:halves}, but the action
of the full automorphism group on all such decompositions remains to be
determined.

Second, Theorem~\ref{thm:extremal} shows that \(14/5\) is the largest slope
of an arrangement of at most \(24\) smooth rational curves on a K3 surface
satisfying~\eqref{eq:nef-condition}, but it leaves open arrangements with more
curves or with curves violating~\eqref{eq:nef-condition}.  In particular, we do
not know whether \(14/5\) is the largest slope among all line arrangements on
the Schur quartic (the discovery computation of Section~\ref{sec:discovery} is
consistent with this, but it is not a proof), nor whether it is the supremum
of the slopes of transversal arrangements of rational curves on K3 surfaces.
By Theorem~\ref{thm:extremal}, an arrangement with a larger slope must contain
more than \(24\) curves or a curve violating~\eqref{eq:nef-condition}.

Third, the lattice \(\Gamma\) of Proposition~\ref{prop:ball} deserves a closer
look.  Since \(D\) is defined over \(\QQ(\sqrt{-3})\) and \(\Gamma\) contains
reflections of order two, it is natural to ask whether \(\Gamma\) is
arithmetic, and how it is related to the known non-cocompact lattices in
\(\mathrm{PU}(2,1)\), such as the Picard modular groups over the Eisenstein
integers.

Finally, the sequence
\[
 64\text{ lines}\supset48\text{ first-kind lines}
 =24+24
\]
raises the question whether comparable decompositions occur on other
line-rich quartic K3 surfaces.  The divisor identity \(D\sim6H\) provides a
useful geometric test for potential analogues.

\appendix

\section{Parametrizations of the 24 lines}\label{app:lines}

Each row below denotes the graph line \(L(a,b;c,d)\)
from~\eqref{eq:graph-line}, over \(K=\QQ(\tau)\) with \(\tau^2=-3\).

\begin{table}[htbp]
\centering
\scriptsize
\setlength{\tabcolsep}{2.4pt}
\renewcommand{\arraystretch}{1.13}
\begin{tabular}{@{}rllll@{\qquad}rllll@{}}
\toprule
\(j\)&\(a\)&\(b\)&\(c\)&\(d\)&
\(j\)&\(a\)&\(b\)&\(c\)&\(d\)\\
\midrule
1&\(1\)&\(0\)&\(0\)&\(1\)&
13&\(\tau/3\)&\(-1+\tau/3\)&\((3+\tau)/6\)&\(-\tau/3\)\\
2&\(-1\)&\(0\)&\(0\)&\(-1\)&
14&\(\tau/3\)&\(-2\tau/3\)&\(-\tau/3\)&\(-\tau/3\)\\
3&\(-1\)&\(0\)&\(0\)&\((1-\tau)/2\)&
15&\(\tau/3\)&\(1+\tau/3\)&\((-3+\tau)/6\)&\(-\tau/3\)\\
4&\(-1\)&\(0\)&\(0\)&\((1+\tau)/2\)&
16&\(-\tau/3\)&\(-1-\tau/3\)&\((3-\tau)/6\)&\(\tau/3\)\\
5&\(1\)&\(0\)&\(0\)&\(-(1+\tau)/2\)&
17&\(-\tau/3\)&\(2\tau/3\)&\(\tau/3\)&\(\tau/3\)\\
6&\(1\)&\(0\)&\(0\)&\((-1+\tau)/2\)&
18&\(-\tau/3\)&\(1-\tau/3\)&\((-3-\tau)/6\)&\(\tau/3\)\\
7&\(-\tau/3\)&\(-1-\tau/3\)&\((-3-\tau)/6\)&\((3-\tau)/6\)&
19&\(-\tau/3\)&\(-1-\tau/3\)&\(\tau/3\)&\((-3-\tau)/6\)\\
8&\(-\tau/3\)&\(2\tau/3\)&\((3-\tau)/6\)&\((3-\tau)/6\)&
20&\(-\tau/3\)&\(2\tau/3\)&\((-3-\tau)/6\)&\((-3-\tau)/6\)\\
9&\(-\tau/3\)&\(1-\tau/3\)&\(\tau/3\)&\((3-\tau)/6\)&
21&\(-\tau/3\)&\(1-\tau/3\)&\((3-\tau)/6\)&\((-3-\tau)/6\)\\
10&\(\tau/3\)&\(-1+\tau/3\)&\((-3+\tau)/6\)&\((3+\tau)/6\)&
22&\(\tau/3\)&\(-1+\tau/3\)&\(-\tau/3\)&\((-3+\tau)/6\)\\
11&\(\tau/3\)&\(-2\tau/3\)&\((3+\tau)/6\)&\((3+\tau)/6\)&
23&\(\tau/3\)&\(-2\tau/3\)&\((-3+\tau)/6\)&\((-3+\tau)/6\)\\
12&\(\tau/3\)&\(1+\tau/3\)&\(-\tau/3\)&\((3+\tau)/6\)&
24&\(\tau/3\)&\(1+\tau/3\)&\((3+\tau)/6\)&\((-3+\tau)/6\)\\
\bottomrule
\end{tabular}
\caption{Graph parameters for the \(24\) selected lines.}\label{tab:lines}
\end{table}

\section{Triple-point coordinates}\label{app:points}

The coordinates are normalized by setting the first nonzero coordinate equal
to \(1\).  The middle column gives the incident line labels from
Table~\ref{tab:lines}.

\begingroup
\small
\renewcommand{\arraystretch}{1.08}
\begin{longtable}{@{}rcl@{}}
\caption{The \(32\) triple points and their incident line triples.}%
\label{tab:triple-points}\\
\toprule
number & incident lines & point in \(\PP^3(K)\)\\
\midrule
\endfirsthead
\toprule
number & incident lines & point in \(\PP^3(K)\)\\
\midrule
\endhead
\bottomrule
\endfoot
1 & $(1,5,6)$ & $[1:0:1:0]$\\
2 & $(1,7,10)$ & $[1:-2:1:-2]$\\
3 & $(1,8,12)$ & $[1:1+\tau:1:1+\tau]$\\
4 & $(1,9,11)$ & $[1:1-\tau:1:1-\tau]$\\
5 & $(2,3,4)$ & $[1:0:-1:0]$\\
6 & $(2,19,23)$ & $[1:1+\tau:-1:-1-\tau]$\\
7 & $(2,20,22)$ & $[1:1-\tau:-1:-1+\tau]$\\
8 & $(2,21,24)$ & $[1:-2:-1:2]$\\
9 & $(3,7,14)$ & $[1:1-\tau:-1:-1-\tau]$\\
10 & $(3,8,13)$ & $[1:-2:-1:-1+\tau]$\\
11 & $(3,9,15)$ & $[1:1+\tau:-1:2]$\\
12 & $(4,10,17)$ & $[1:1+\tau:-1:-1+\tau]$\\
13 & $(4,11,16)$ & $[1:-2:-1:-1-\tau]$\\
14 & $(4,12,18)$ & $[1:1-\tau:-1:2]$\\
15 & $(5,13,19)$ & $[1:1-\tau:1:-2]$\\
16 & $(5,14,21)$ & $[1:1+\tau:1:1-\tau]$\\
17 & $(5,15,20)$ & $[1:-2:1:1+\tau]$\\
18 & $(6,16,22)$ & $[1:1+\tau:1:-2]$\\
19 & $(6,17,24)$ & $[1:1-\tau:1:1+\tau]$\\
20 & $(6,18,23)$ & $[1:-2:1:1-\tau]$\\
21 & $(7,16,19)$ & $[1:0:\tfrac{-\tau}{3}:\tfrac{-3-\tau}{3}]$\\
22 & $(7,18,20)$ & $[1:1+\tau:-\tau:0]$\\
23 & $(8,16,21)$ & $[1:1-\tau:-\tau:0]$\\
24 & $(8,17,20)$ & $[1:0:\tfrac{-\tau}{3}:\tfrac{2\tau}{3}]$\\
25 & $(9,17,19)$ & $[1:-2:-\tau:0]$\\
26 & $(9,18,21)$ & $[1:0:\tfrac{-\tau}{3}:\tfrac{3-\tau}{3}]$\\
27 & $(10,13,22)$ & $[1:0:\tfrac{\tau}{3}:\tfrac{-3+\tau}{3}]$\\
28 & $(10,15,23)$ & $[1:1-\tau:\tau:0]$\\
29 & $(11,13,24)$ & $[1:1+\tau:\tau:0]$\\
30 & $(11,14,23)$ & $[1:0:\tfrac{\tau}{3}:\tfrac{-2\tau}{3}]$\\
31 & $(12,14,22)$ & $[1:-2:\tau:0]$\\
32 & $(12,15,24)$ & $[1:0:\tfrac{\tau}{3}:\tfrac{3+\tau}{3}]$\\
\end{longtable}
\endgroup

\section*{Acknowledgments}

The authors gratefully acknowledge PSNC for access to the Eagle cluster, as well as the Faculty of Mathematics and Computer Science for access to the local computing cluster. We also acknowledge the assistance of GPT-5.6 in identifying the relevance of fractional linear programming, a key insight that made the discovery of the counterexample possible. The authors remain solely responsible for the verification, interpretation, and presentation of the results.

\section*{Funding}
 Piotr Pokora is supported by the National Science Centre (Poland) Sonata Bis Grant  \textbf{2023/50/E/ST1/00025}. For the purpose of Open Access, the author has applied a CC-BY public copyright license to any Author Accepted Manuscript (AAM) version arising from this submission.

\end{document}